\documentclass[letterpaper,11pt,twoside,keywordsasfootnote,addressatend,noinfoline]{article}
\usepackage{fullpage}
\usepackage[english]{babel}
\usepackage{amssymb}
\usepackage{amsmath}
\usepackage{theorem}
\usepackage{epsfig}
\usepackage{subfigure}
\usepackage{imsart}
\usepackage{verbatim}
\usepackage{color}

\theorembodyfont{\slshape}
\newtheorem{theor}{Theorem}

\newtheorem{propo}[theor]{Proposition}
\newtheorem{lemma}[theor]{Lemma}

\newtheorem{defin}[theor]{Definition}

\newenvironment{proof}{\noindent{\scshape Proof.}}{\hspace{2mm} $\square$}
\newenvironment{demon}[1]{\noindent{\scshape Proof of Proposition \ref{#1}.}}{\hspace{2mm} $\square$}

\newcommand{\Z}{\mathbb{Z}}

\newcommand{\N}{\mathbb{N}}
\newcommand{\D}{\mathbb{D}}
\newcommand{\ind}{\mathbf{1}}
\newcommand{\ep}{\epsilon}

\begin{document}
\begin{frontmatter}
\title{Consensus in the two-state Axelrod model}
\runtitle  {Consensus in the two-state Axelrod model}
\author{Nicolas Lanchier\thanks{Research supported in part by NSF Grant DMS-10-05282} and
       Jason Schweinsberg\thanks{Research supported in part by NSF Grant DMS-08-05472}}
\runauthor {N. Lanchier and J. Schweinsberg}
\address   {School of Mathematical and Statistical Sciences, \\ Arizona State University, \\ Tempe, AZ 85287, USA.}
\address   {Department of Mathematics, \\ University of California, San Diego, \\ La Jolla, CA 92093-0112, USA.}

\maketitle

\begin{abstract} \ \
 The Axelrod model is a spatial stochastic model for the dynamics of cultures which, similarly to the voter model, includes
 social influence, but differs from the latter by also accounting for another social factor called homophily, the
 tendency to interact more frequently with individuals
 who are more similar.
 Each individual is characterized by its opinions about a finite number of cultural features, each of which can assume the
 same finite number of states.
 Pairs of adjacent individuals interact at a rate equal to the fraction of features they have in common, thus modeling
 homophily, which results in the interacting pair having one more cultural feature in common, thus modeling social influence.
 It has been conjectured based on numerical simulations that the one-dimensional Axelrod model clusters when the number of
 features exceeds the number of states per feature.
 In this article, we prove this conjecture for the two-state model with an arbitrary number of features.
\end{abstract}



\end{frontmatter}

\section{Introduction}

\indent There has been in the past decade a rapidly growing interest in agent-based models in an attempt to understand the
 long-term behavior of complex social systems.
 These models are characterized by heuristic rules that govern the outcome of an interaction between two agents, and a
 graphical structure, modeling either physical space or a social network, that encodes the pairs of agents that may interact due
 to, e.g., geographical proximity or friendship.
 The main objective of research in this field is to deduce the macroscopic behavior that emerges from the microscopic rules, which
 also depends on the structure of the network of interactions.
 The mathematical term for agent-based models is interacting particle systems though, as scientific fields, the former involves
 numerical simulations whereas the latter is based on rigorous mathematical analyses.
 While there is a common effort from sociologists, economists, psychologists, and statistical physicists to understand such
 models, interacting particle systems of interest in social sciences have been so far essentially ignored by mathematicians, with
 the notable exception of the voter model.
 This paper is motivated by this lack of analytical results and continues the study initiated in \cite{lanchier_2011} for one
 of the most popular models of social dynamics: the Axelrod model \cite{axelrod_1997}.
 The effort to collect analytical results is mainly justified by the fact that stochastic spatial simulations are generally
 difficult to interpret.
 This is especially true for the Axelrod model which, in contrast with the voter model, has a number of absorbing states that
 grows exponentially with the size of the network, and for which simulations of the finite system can freeze in atypical
 configurations, thus exhibiting behaviors which are not symptomatic of the long-term behavior of their infinite counterpart.

\indent The Axelrod model \cite{axelrod_1997} has been proposed by political scientist Robert Axelrod as a stochastic model for
 the dissemination of culture.
 The heuristic microscopic rules include two important social factors:
 homophily, which is the tendency of individuals to interact more frequently with individuals who are more similar, and
 social influence, which is the tendency of individuals to become more similar when they interact.
 Note that the voter model \cite{clifford_sudbury_1973, holley_liggett_1975} accounts for the latter but not for the former:
 individuals are characterized by one of two competing opinions which they update at a constant rate by mimicking one of their
 neighbors chosen uniformly at random.
 In particular, any two individuals in the voter model either totally agree or totally disagree, which prevents homophily from being incorporated in the model.
 In contrast, individuals in the Axelrod model are characterized by a vector, also called a culture, that consists of $F$ coordinates,
 called cultural features, each of which assumes one of $q$ possible states.
 Homophily can thus be naturally modeled in terms of a certain cultural distance between two individuals: pairs of neighbors interact
 at a rate equal to the fraction of features they have in common.
 Social influence is then modeled as follows:
 each time two individuals interact, one of the cultural features for which the interacting pair disagrees (if any) is chosen uniformly
 at random, and the state of one of both individuals is set equal to the state of the other individual for this cultural feature.
 More formally, the Axelrod model on the infinite one-dimensional lattice, which is the network of interactions considered in this
 paper, is the continuous-time Markov chain whose state space consists of all spatial configurations
 $$ \eta : \Z \ \longrightarrow \ \{1, 2, \ldots, q \}^F. $$
 To describe the dynamics of the Axelrod model, we let
 $$ F (x, y) \ = \ \frac{1}{F} \ \sum_{i = 1}^F \ \ind \{\eta (x, i) = \eta (y, i) \}, $$
 where $\eta (x, i)$ refers to the $i$th coordinate of the vector $\eta (x)$, denote the fraction of cultural features vertex $x$
 and vertex $y$ have in common.
 In addition, we introduce the operator $\sigma_{x, y, i}$ on the set of spatial configurations defined by
 $$ (\sigma_{x, y, i} \,\eta) (z, j) \ = \ \left\{\hspace{-4pt} \begin{array}{ll}
     \eta (y, i) & \hbox{if} \ z = x \ \hbox{and} \ j = i \vspace{2pt} \\
     \eta (z, j) & \hbox{otherwise} \end{array} \right. \quad \hbox{for} \ x, y \in \Z \ \hbox{and} \ i \in \{1, 2, \ldots, q \}. $$
 In other words, configuration $\sigma_{x, y, i} \,\eta$ is obtained from configuration $\eta$ by setting the $i$th feature of the
 individual at vertex $x$ equal to the $i$th feature of the individual at vertex $y$ and leaving the state of all the other features
 in the system unchanged.
 The dynamics of the Axelrod model are then described by the Markov generator $L$ defined on the set of cylinder functions by
 $$ Lf (\eta) \ = \ \sum_{|x - y| = 1} \ \sum_{i = 1}^F \ \frac{1}{2F} \
                    \bigg[\frac{F (x, y)}{1 - F (x, y)} \bigg] \ \ind \{\eta (x, i) \neq \eta (y, i) \} \ [f (\sigma_{x, y, i} \,\eta) - f (\eta)]. $$
 Note that the expression of the Markov generator indicates that the conditional rate at which the $i$th feature of vertex $x$ is
 set equal to the $i$th feature of vertex $y$ given that these two vertices are nearest neighbors that disagree on their $i$th
 feature can be written as
 $$ \frac{1}{2F} \ \bigg[\frac{F (x, y)}{1 - F (x, y)} \bigg] \ = \ F (x, y) \ \times \ \frac{1}{F \,(1 - F (x, y))} \ \times \ \frac{1}{2} $$
 which, as required, equals the fraction of features both vertices have in common, which is the rate at which the vertices interact, times
 the reciprocal of the number of features for which both vertices disagree, which is the probability that any of these features
 is the one chosen to be updated, times the probability one half that vertex $x$ rather than vertex $y$ is chosen to be updated.

\indent The main question about the Axelrod model is whether or not the population converges to a consensus when starting from a random configuration.
 For simplicity, we assume that the initial cultures at different sites are independent and identically distributed, and that at
 a given site, each of the $q^F$ possible initial cultures appears with the same probability.
 The term ``consensus'' is defined mathematically in terms of clustering
 of the infinite system: the model is said to cluster if
 $$ \lim_{t \to \infty} \ P \,(\eta_t (x, i) = \eta_t (y, i)) \ = \ 1 \quad \hbox{for all} \ x, y \in \Z \ \hbox{and} \ i \in \{1, 2, \ldots, q \} $$
 and is said to coexist otherwise.
 The dichotomy between clustering and coexistence for the finite model is unclear since, as mentioned above, the finite system can
 hit an absorbing state in which different cultures are present even though its infinite counterpart clusters.
 In order to characterize the transition between the two regimes, Vilone \emph{et al} \cite{vilone_vespignani_castellano_2002}
 considered the random variable $s_{\max}$ which refers to the length of the largest interval in which all individuals share the
 same culture in the absorbing state hit by the finite system, and distinguished between the two regimes depending on whether
 the expected value of this random variable scales like the population size or is uniformly bounded.
 Denoting the population size by $N$, their spatial simulations suggest that
 $$ \begin{array}{rcl}
    \lim_{N \to \infty} \ E \,(N^{-1} s_{\max}) > 0 \ \hbox{for} \ F > q & \hbox{and} &
    \lim_{N \to \infty} \ E \,(s_{\max}) < \infty   \ \hbox{for} \ F < q \end{array} $$
 so we conjecture clustering when $F > q$ and coexistence when $F < q$ for the one-dimensional infinite system.
 The mathematical analysis of the Axelrod model initiated in \cite{lanchier_2011} strongly suggests the coexistence part of the
 conjecture for a certain subset of the parameter region.
 More precisely, letting $n_c$ denote the number of cultural domains in the absorbing state hit by
 the finite system, it is proved based on duality-like techniques and a coupling with a simple urn problem that
 $$ \begin{array}{rl} \lim_{N \to \infty} \ E \,(N^{-1} \,n_c) > 0 \ \hbox{for} \ F < c \times q & \hbox{where $c \approx 0.567$ satisfies $c = e^{-c}$}. \end{array} $$
 It is also proved that the infinite system clusters in the critical case $F = q = 2$.
 In this paper, we extend this result to all values of the number of features when the number of states per feature is again equal
 to two, which proves the clustering part of the conjecture stated above when $q = 2$.
 More precisely, we have the following theorem.
\begin{theor} --
\label{thm}
 The $F$-feature 2-state Axelrod model on $\Z$ clusters, starting from the random initial configuration in which the cultures at
 different sites are independent and identically distributed and at a given site, each of the $2^F$ possible initial cultures appears with the same probability.
\end{theor}
 The proof when $F = q = 2$ is carried out in \cite{lanchier_2011} based on duality techniques for the voter model through the existence
 of a natural coupling between the two-feature two-state Axelrod model and the voter model obtained by identifying cultures with no
 feature in common.
 This coupling, however, fails for any other values of the parameters, so a different approach is needed to extend the
 result to a larger number of features.
 The first step is to construct a coupling between the two-state Axelrod model and a certain collection of non-independent systems
 of annihilating symmetric random walks that keep track of the disagreements between nearest neighbors.
 Clustering of the Axelrod model is equivalent to extinction of these systems of annihilating random walks.
 The proof of the latter is inspired by a symmetry argument introduced by Adelman \cite{adelman_1976} which is combined with certain
 parity properties of the collection of non-independent systems of random walks.

\pagebreak

\section{Systems of annihilating random walks}

\begin{figure}[t]
\centering
\mbox{\subfigure{\epsfig{figure = 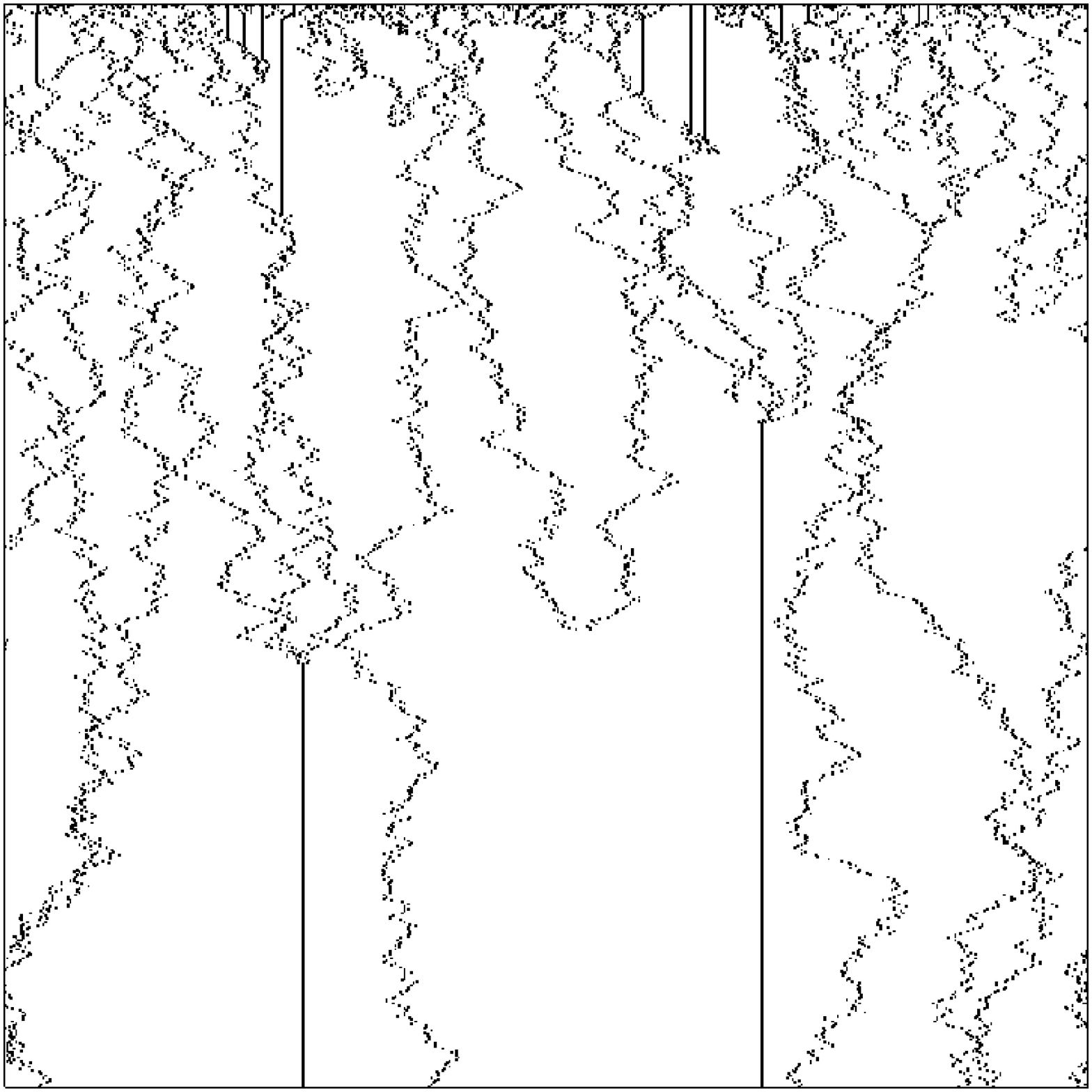, width = 210pt}} \hspace{5pt}
      \subfigure{\epsfig{figure = 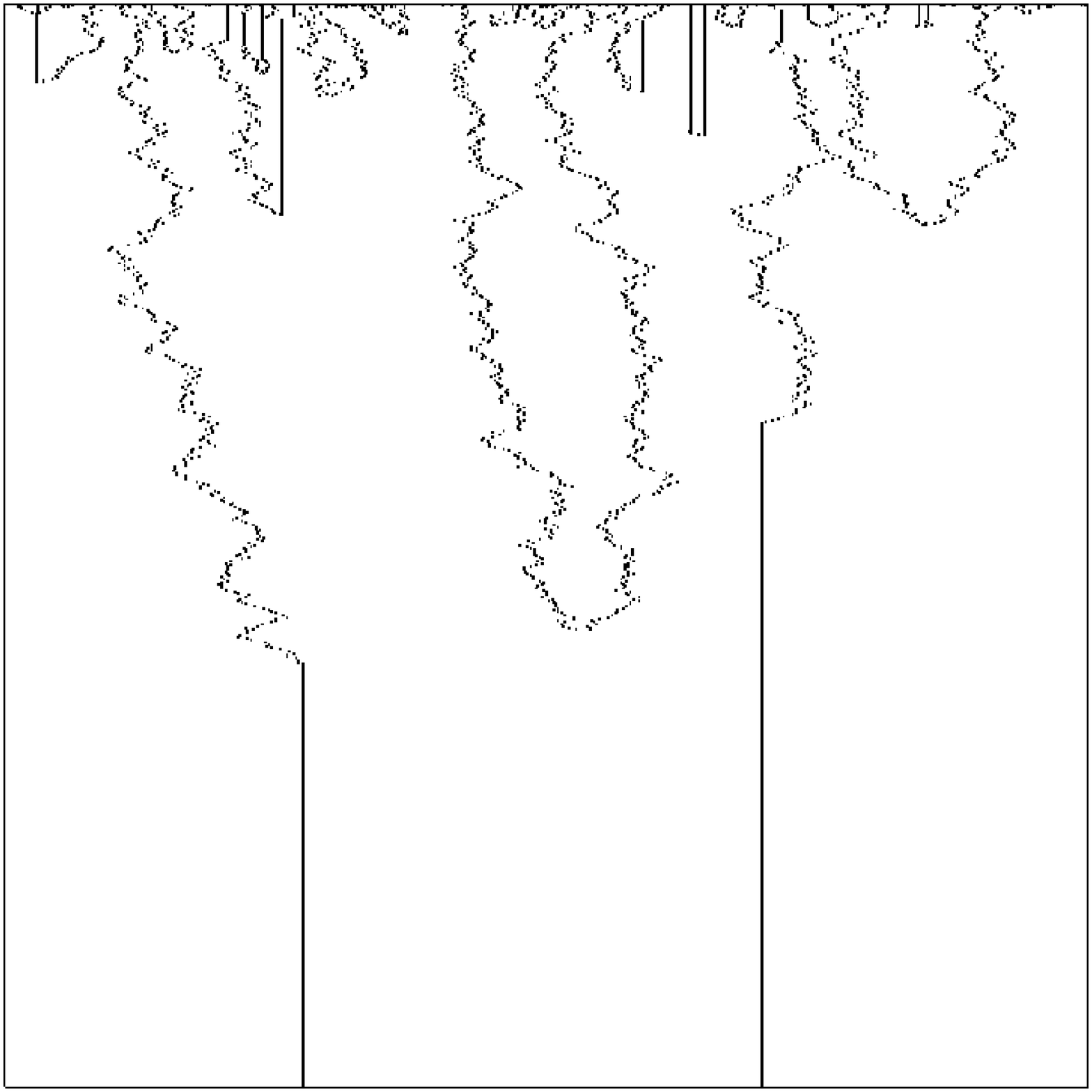, width = 210pt}}}
\caption{\upshape{Pictures of the system of coalescing random walks obtained from a realization of the two-state Axelrod model on the one-dimensional
 torus with 600 vertices with $F = 3$ features.
 The left picture gives a superposition of all three levels of random walks while the first level only is extracted in the right picture.}}
\label{fig:AxM}
\end{figure}

\indent In this section, we represent the Axelrod model on $\Z$ by a particle system that keeps track of the interfaces between cultural
 domains, thus looking at the disagreements along the edges of the graph rather than the actual cultures on the vertices.
 This approach is motivated by the fact that consensus in the Axelrod model is equivalent to the extinction of its interfaces.
 In order to obtain a well-defined Markov process, it is necessary to keep track of the features for which neighbors disagree rather
 than simply the number of these features.
 Therefore, we think of each edge of the graph as having $F$ levels, and place a particle on edge $e = (x, y)$ at level $i$ if and only
 if vertex $x$ and vertex $y$ disagree on their $i$th feature.
 That is, we define
 $$ \xi_t (u, i) \ = \ \ind \,\{\eta_t (u - 1/2, \,i) \neq \eta_t (u + 1/2, \,i) \} \quad \hbox{for all} \ u \in \D := \Z + 1/2 $$
 and place a particle at site $u \in \D$ at level $i$ whenever $\xi_t (u, i) = 1$.
 To study the dynamics of this system and the rates at which particles jump, it will be useful to also keep track of the number of particles
 per site so we introduce
 $$ \zeta_t (u) \ = \ \sum_{i = 1}^F \ \xi_t (u, i) \quad \hbox{for all} \ u \in \D $$
 and call $u$ a $j$-site whenever it has a total of $j$ particles: $\zeta_t (u) = j$.
 To understand the evolution, we first observe that when a particle jumps, it moves right or left with equal probability, unless another particle
 already occupies the site on which the particle tries to jump in which case the particles annihilate each other. 
 Thus, these processes induce a collection of $F$ non-independent systems of annihilating symmetric random walks.
 The symmetry is due to the fact that, when two neighbors interact, each of them is equally likely to be the one chosen to be updated.
 Also, the reason why a collision between two particles results in an annihilation is
that when two individuals disagree with a third one on
 a given feature, these two individuals must agree on this feature, which happens when $q = 2$.
(Note however that for larger values of $q$ collisions between particles would result in either an annihilating event or a coalescing event
 depending on the configuration of the underlying Axelrod model.)
 Even though the evolution of the particle system at a single level is somewhat reminiscent of the evolution of the interfaces of the
 one-dimensional voter model, it is in fact much more complicated due to the presence of strong dependencies among the different levels.
 These dependencies result from the inclusion of homophily in the model, which implies that particles jump at varying rates.
 More precisely, since two adjacent vertices that disagree on exactly $j$ of their features interact at rate $1 - j/F$, the fraction of
 features they share, and the site between these vertices is a $j$-site, given that $u$ is a $j$-site each particle at site $u$ jumps at rate
 $$ r (j) \ = \ \frac{1}{j} - \frac{1}{F} \quad \hbox{for} \ j \neq 0. $$
 This indicates that the motion of the particles is slowed down at sites that contain more particles, with the dynamics being frozen at sites
 with $F$ particles.
In the following sections, we will call particles frozen or active depending on whether these particles are
 located at an $F$-site or not, respectively.
 We refer the reader to Figure \ref{fig:AxM} for simulation pictures of the systems of annihilating random walks when $F = 3$.

\section{Showing that the process can not become frozen}

\indent In this section, we prove that no site can remain an $F$-site forever, which is the key to proving consensus in the two-state
 Axelrod model.
 From the point of view of the particle system described in the previous section, this means that if some site $u$ is completely filled
 with particles at some time, then eventually another particle will jump onto the site $u$, annihilating one of the frozen particles
 on the site and making the other $F - 1$ particles active.

\begin{propo} --
\label{prop:symmetry}
 Assume that site $u \in \D$ is a $F$-site at time $t$. Then,
 $$ T \ := \ \inf \,\{s > t : \zeta_s (u) \neq F \} \ < \ \infty \quad a.s. $$
\end{propo}

\noindent It will be useful in the proof of the proposition to construct graphically the particle system described in the previous
 section using the following collections of independent Poisson processes and random variables:
 for each pair of site and feature $(v, i) \in \D \times \{1, 2, \ldots, F \}$,
\begin{itemize}
 \item we let $(N_{v, i} (t) : t \geq 0)$ be a rate one Poisson process, \vspace{4pt}
 \item we denote by $T_{v, i} (n)$ its $n$th arrival time: $T_{v, i} (n) = \inf \,\{t : N_{v, i} (t) = n \}$, \vspace{4pt}
 \item we let $(B_{v, i} (n) : n \geq 1)$ be a collection of independent Bernoulli variables with
  $$ P \,(B_{v, i} (n) = + 1) \ = \ P \,(B_{v, i} (n) = - 1) \ = \ 1/2, $$
 \item and we let $(U_{v, i} (n) : n \geq 1)$ be a collection of independent Uniform (0, 1) random variables.
\end{itemize}
 The system of annihilating random walks is constructed as follows.
 At time $t = T_{v, i} (n)$, we draw an arrow labeled $i$ from site $v$ to site $v + B_{v, i} (n)$ to indicate that if
 $$ \xi_{t-} (v, i) \ = \ 1 \quad \hbox{and} \quad U_{v, i} (n) \ \leq \ r (\zeta_{t-}(v)) \ = \ \frac{1}{\zeta_{t-}(v)} - \frac{1}{F} $$
 then the particle at site $v$
at level $i$ jumps to site $v + B_{v, i} (n)$.

\indent The above construction can be extended naturally to any subgraph $G$ of the lattice by using the same collections of independent
 Poisson processes and random variables and killing all the particles that jump onto a site $v$ which is not the center of an edge of the graph.
 Consider now the case in which the graph $G$ is the one induced by the vertex set $\N$.
 Suppose the initial configuration is such that the left-most edge has $F$ particles, one at every level, while every level of every
 other edge independently has a particle at time zero with probability $1/2$. Let
 $$p \ := \ P \,(\hbox{the left-most edge has $F$ particles at all times}). $$
 That is, $p$ is the probability that no particle ever tries to jump onto one of the particles on the left-most edge.
 We will later see that $p = 0$.

\indent Returning now to the setting of Proposition \ref{prop:symmetry}, fix a site $u \in \D$ and a time $t > 0$, and suppose that
 particles at that space time point are frozen: $\zeta_t (u) = F$.
 We will consider only the sites to the right of $u$ and show that eventually some particle from the right must jump onto $u$,
 unless the site $u$ has already been hit from the left.
 Lemma \ref{lem:symmetry} below will allow us to break the process into stages, and give a lower bound for the probability that a particle
 at $u$ is annihilated at each stage.
 We start by proving Lemma \ref{lem:break} which is a key preliminary result.

\begin{defin} --
\label{def:active}
 The interval $\{u, u + 1, \ldots, v \} \subset \D$ is said to be active at time $t$ if the numbers of particles it contains at two
 different levels differ in their parity, i.e.,
 $$ \sum_{w = u}^v \ \xi_t (w, i) \ \neq \ \sum_{w = u}^v \ \xi_t (w, j) \mod 2 \quad\hbox{for some} \ i \neq j. $$
\end{defin}

\begin{lemma} --
\label{lem:break}
 Assume that $\{u, \ldots, v \}$ is active at time $t$ and $\zeta_t (u) = \zeta_t (v) = F$. Then,
 $$ T \ := \ \inf \,\{s > t : \zeta_s (u) + \zeta_s (v) \neq 2 F \} \ < \ \infty \quad a.s. $$
\end{lemma}
\begin{proof}
 Seeking a contradiction, we assume that $\zeta_s (u) = \zeta_s (v) = F$ for all $s > t$.
 Under this assumption, the parity of the number of particles at each level between site $u$ and site $v$ is preserved since particles
 annihilate by pair.
 In particular, $\{u, \ldots, v \}$ is active at every later time, which implies that this interval contains at least one site which is
 neither a 0-site nor an $F$-site, and thus at least one active particle, at every later time.
 Since this particle jumps at a positive rate, it must hit one of the boundaries $u$ or $v$ in an almost surely finite time, which leads
 to a contradiction.
\end{proof}

\begin{lemma} --
\label{lem:symmetry}
 There exists a sequence of times $t = t_0 < t_1 < t_2 < \dots < \infty$ such that, if $A_k$ denotes the event that at some time
 $s \in (t_{k - 1}, t_k]$, a particle at site $u$ is annihilated, then
\begin{equation}
\label{Aiprob}
 P \,\bigg(A_k \ \Big| \ \bigcap_{j = 1}^{k - 1} \,A_k^c \bigg) \ \geq \ \frac{p}{8} \quad \hbox{for all} \ k > 0.
\end{equation}
\end{lemma}
\begin{proof}
 The proof relies in part on delicate symmetry arguments and is modeled after the construction of Adelman \cite{adelman_1976}.
 We refer the reader to Figure \ref{fig:symmetry} for a picture of this construction in our context.
 We must analyze in detail the process at each stage.
 For $w \in \D$, $w > u$, let
\begin{align}
 {\cal F}_w (t) \ = \ \sigma \,((N_{v, i} (s) & : 0 \leq s \leq t), (B_{v,i}(n) : 1 \leq n \leq N_{v, i} (t)), \nonumber \\
  & (U_{v, i} (n) : 1 \leq n \leq N_{v, i} (t)) : u \leq v \leq w, \,1 \leq i \leq q) \nonumber
\end{align}
 be the $\sigma$-field generated by the graphical representation of the systems of random walks over the spatial interval
 $\{u, u + 1, \dots, w \}$ through time $t$.
 We will prove (\ref{Aiprob}) by induction on $k$.
 As part of the induction hypothesis, we will assume that at the beginning of the $k$th stage of the construction, we have a
 site $v_k$ such that the following two conditions hold:
\begin{enumerate}
 \item[H1.] $N_{v_k, i} \,(t_{k - 1}) = N_{v_k + 1, i} \,(t_{k - 1}) = N_{v_k + 2, i} \,(t_{k - 1}) = 0$ for $i = 1, 2, \dots, F$, \vspace{4pt}
 \item[H2.] $A_1, A_2, \dots, A_{k-1} \in {\cal F}_{v_k}(t_{k-1})$.
\end{enumerate}
 For $k = 1$, condition H2 is trivial while H1 can be satisfied by choosing $t_0 = t$ and
 $$ v_1 \ = \ \min \,\{v > u : N_{v, i} (t) = N_{v + 1, i} (t) = N_{v + 2, i} (t) = 0 \ \hbox{for} \ i = 1, 2, \dots, F \}. $$
 To prove \eqref{Aiprob} by induction, we will show that
\begin{equation}
\label{AkFprob}
 P \,(A_k \,| \,{\cal F}_{v_k} (t_{k - 1})) \ \geq \ \frac{p}{8}
\end{equation}
 and that $t_k$ and $v_{k + 1}$ can be chosen to satisfy conditions H1 and H2.

\indent Condition H1 above implies that, in the graphical representation, there is no arrow starting at either site $v_k$ or site $v_k + 1$
 by time $t_{k - 1}$ from which it follows that
\begin{itemize}
 \item particles starting to the right of $v_k$ do not reach $v_k$ by time $t_{k - 1}$, and \vspace{4pt}
 \item particles starting in $\{u, u + 1, \dots, v_k \}$ do not reach $v_k + 1$ until after time $t_{k - 1}$.
\end{itemize}
 Note that the assumption $N_{v_k + 2, i} \,(t_{k - 1}) = 0$ is not necessary at this stage of the proof but it will be useful later in the
 proof to obtain a lower bound of the probability that a certain interval is active.
 We will partition the region to the right of site $v_k + 2$ into intervals of equal length, each containing a total of $v_k - u + 3$ sites.
 Specifically, for each positive integer $j$, we let
 $$ B_{j, k} = \{j v_k - (j - 1) u + 3j, \dots, (j + 1) v_k - ju + 3j + 2 \}. $$
 Also, we define the interval $B_{0, k} = \{u, u + 1, \ldots, v_k + 2 \}$ as depicted in Figure \ref{fig:symmetry}.
 Now let $J_k$ be the smallest positive integer $j$ such that the following two conditions hold:
\begin{itemize}
\item we have $N_{w, i} (t_{k - 1}) = 0$ for all $w \in B_{j, k}$ and $i = 1, 2, \dots, F$, and \vspace{4pt}
\item we have either
 $$ \begin{array}{rcl}
    \hbox{translation}   & : & \xi_0 (j v_k - (j - 1) u + 3j + m, i) \ = \ \xi_{t_{k - 1}} (u + m, i) \vspace{4pt} \\
    \hbox{or reflection} & : & \xi_0 ((j + 1) v_k - ju + 3j + 2 - m, i) \ = \ \xi_{t_{k - 1}} (u + m, i) \end{array} $$
 for $m = 0, 1, \dots, v_k - u + 1$ and $i = 1, 2, \dots, F$.
\end{itemize}
 The second requirement ensures that either the configuration in $B_{J_k, k}$ excluding its right-most site at time zero is the same as the
 configuration between $u$ and $v_k+1$ at time $t_{k - 1}$, or else the configuration in $B_{J_k, k}$ excluding its left-most site at time zero is
 the mirror image of the configuration between $u$ and $v_k+1$ at time $t_{k - 1}$.
 In the latter case, we say that reflection occurs, an event which, due to obvious symmetry, has probability $1/2$.

\indent We now consider the process conditioned on the event $J_k = j$ as well as the event that reflection occurs.
 We also, for now, truncate the process at the right edge of $B_{j, k}$, so the right-most site is
 $$ r_k \ := \ (j + 1) \,v_k - ju + 3j + 2 $$
 and just consider the evolution of the process between the sites $u$ and $r_k$. Let
 $$ \ell_k \ := \ j \,v_k - (j - 1) \,u + 3j $$
 be the left-most site in $B_{j, k}$.
 We make the following observations:
\begin{itemize}
\item Because no particles in the interval $B_{j, k}$ can jump before time $t_{k - 1}$, the left-most and the right-most $v_k - u + 1$ sites are mirror
 images of each other at time $t_{k - 1}$, i.e.,
 $$ \xi_{t_{k - 1}} (u + m, i) \ = \ \xi_{t_{k - 1}} (r_k - m, i) \quad \hbox{for} \ \ m = 0, 1, \dots, v_k - u + 1 \ \hbox{and all} \ i. $$
\item Because particles at the sites $v_k$, $v_k + 1$, $v_k + 2$, $\ell_k$, and $\ell_k + 1$ do not jump before time $t_{k-1}$, conditional on the configuration at
 the left-most and the right-most $v_k - u + 1$ sites at time $t_{k - 1}$, the sites in between evolve independently from time $0$ until time $t_{k-1}$.
 The law of their evolution is the same as the law of the original process, modified so that particles at the first two sites and the last two sites are not
 permitted to jump and conditioned on the blocks $B_{i, k}$ for $i = 1, \dots, j-1$ failing to satisfy the conditions of either translation or reflection.
 As a result of the way $J_k$ is defined, the law of this process is therefore the same as the law of its mirror image. \vspace{4pt}
\item Because particles at site $v_k + 1$ and site $\ell_k + 1$ have no opportunity to interact with other particles before time $t_{k - 1}$, on these
 sites at each level there is a particle independently with probability $1/2$.
 Consequently, the probability that the number of particles at two given levels of these two sites have the same parity is equal to $1/2$.
\end{itemize}
By the last observation above, if a reflection occurs, then the interval from $u$ to $r_k$ is active in the sense
 of Definition \ref{def:active} with probability at least $1/2$.
 Furthermore, by the symmetry noted in the second observation above, the first change either to site $u$ or site $r_k$ happens to site $u$ with
 probability $1/2$.
 Hence, if $D_k$ denotes the event that eventually there is a change to site $u$, then
 $$ P \,(D_k \,| \,{\cal F}_{v_k} (t_{k - 1})) \ \geq \ \bigg(\frac{1}{2} \bigg)^3 \ = \ \frac{1}{8}. $$
 To complete the construction, we let $t_k = t_{k - 1}$ in case a reflection does not occur or the interval from $u$ to $r_k$ is not active.
 Otherwise, we define time $t_k$ to be the first time at which there is a change either to site $u$ or to site $r_k$ in this system.

\indent We now return to the case in which only the sites to the left of $u$ are discarded, which means there is a possibility that some particle could
 jump onto $r_k$ from the right.
 In this case, conditional on the event of a reflection, which implies that the particles at $r_k$ are frozen until at least $t_k$, and thinking of
 site $r_k$ as the left-most edge in the definition of $p$, we see that there is a probability larger than $p$ that no particle will jump from $r_k + 1$
 to $r_k$ before time $t_k$.
 We deduce that there is probability at least $p/8$ that site $u$ will change by time $t_k$, which gives \eqref{AkFprob}.
 Finally, we let
 $$ v_{k + 1} \ = \ \min \,\{j > r_k : N_{v, i} (t_k) = N_{v + 1,i} (t_k) = N_{v + 2, i} (t_k) = 0 \ \hbox{for} \ i = 1, 2, \dots, F \}. $$
 Because no particle to the right of $v_{k + 1}$ can reach $v_k$ by time $t_k$, we have $A_k \in {\cal F}_{v_{k + 1}} (t_k)$, which completes the
 proof of the lemma.
\end{proof} \\

\begin{figure}[t!]
 \centering
 \scalebox{0.55}{\input{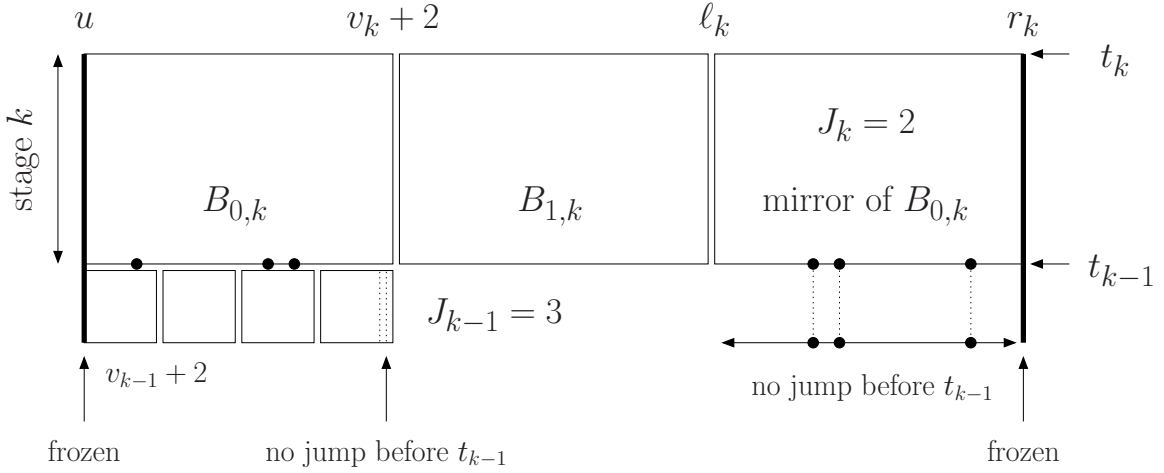}}
 \caption{\upshape{Picture related to the proof of Lemma \ref{lem:symmetry}.}}
\label{fig:symmetry}
\end{figure}

\begin{demon}{prop:symmetry}
 Consider the system in which sites to the left of $u$ are removed.
 We proceed by contradiction by assuming that there is a positive probability that the site $u$ never changes after time $t$.  We will show that this
 implies that $p > 0$.
 Thus, \eqref{Aiprob} will imply that
 $$ P \,\bigg(\bigcup_{k = 1}^{\infty} \ A_k \bigg) \ = \ 1 $$
 which means that with probability one, eventually some particle at site $u$ will be annihilated.
 This contradiction will imply the result.

\indent In the case $t = 0$, the probability that the site $u$ never changes after time $t$ is exactly $p$, by definition.
 We still need to show that if site $u$ becomes an $F$-site at $t > 0$, then having a positive probability that the site never changes after
 time $t$ would still imply that $p > 0$.
 The strategy will be to argue that with positive probability, up to time $t$ the process behaves within a finite region in such a way that makes
 it possible to reduce to the $t = 0$ case.

\indent Because, with probability one, there are infinitely many sites $v > u$ such that
 $$ N_{v, i} (t) \ = \ N_{v + 1, i} (t) \ = \ 0 \quad \hbox{for} \ i = 1, 2, \dots, F, $$
 there exists a site $v > u$ such that
 $$ P \,(\zeta_s (u) = F \ \hbox{for all} \ s > t \ \hbox{and} \ N_{v, i} (t) = N_{v + 1,i} (t) = 0 \ \hbox{for} \ i = 1, 2, \dots, F) \ > \ 0. $$ 
 The event $N_{v, i} (t) = N_{v + 1,i} (t) = 0$ for $i = 1, 2, \dots, F$ implies that there is no arrow connecting site $v$ and site $v + 1$ by
 time $t$.  Therefore conditional on this event, the evolution of the process on the interval $\{u, u + 1, \dots, v \}$ is independent
 of its evolution on $\{v + 1, v + 2, \dots \}$ up to time $t$.

\indent Because there are only finitely many possible configurations for the sites $u + 1, \dots, v - 1$, there must exist numbers
 $c_{w, i} \in \{0, 1 \}$ for $u < w \leq v$ and $i = 1, 2, \dots F$ such that
\begin{align}
 P \,(\zeta_s (u) = F \ \hbox{for all} \ s > t, \ \xi_t (w, i) & = c_{w, i} \ \hbox{for} \ u < w \leq v \ \hbox{and} \ i = 1, 2, \dots, F, \nonumber \\ 
     & \hbox{and} \ N_{v, i} (t) = N_{v + 1,i} (t) = 0 \ \hbox{for} \ i = 1, 2, \dots, F) \ > \ 0. \nonumber
\end{align}
 Let $c_{u,i} = 1$ for $i = 1, 2, \dots, F$.
 Clearly there is a positive probability that
 $$ \begin{array}{rl}
    \xi_0 (w, i) \ = \ c_{w,i} & \hbox{for} \ w = u, u + 1, \dots, v \ \hbox{and} \ i = 1, 2, \dots, F \ \hbox{and} \vspace{4pt} \\
     N_{w, i} (t) \ = \ 0 & \hbox{for} \ w = u, u + 1, \dots, v \ \hbox{and} \ i = 1, 2, \dots, F \end{array} $$
 so it follows from the conditional independence noted above that
\begin{align}
 P \,(\zeta_s (u) = F \ \hbox{for all} \ s > t, \ \xi_0 (w, i) & = c_{w, i} \ \hbox{for} \ u \leq w \leq v \ \hbox{and} \ i = 1, 2, \dots, F, \nonumber \\
    & \hbox{and} \ N_{w, i} (t) = 0 \ \hbox{for} \ w = u, u + 1, \dots, v + 1) \ > \ 0. \nonumber
\end{align}
 However, the probability on the left-hand side is at most
 $$ P \,(\zeta_s (u) = F \ \hbox{for all} \ s > 0 \ \hbox{and} \ \xi_0 (u, i) = 1 \ \hbox{for} \ i = 1, 2, \dots, F), $$
 which in turn is at most $p$.
 It follows that $p > 0$.
\end{demon}

\section{Extinction of the particles}

\indent In this section, we prove almost sure extinction of the systems of annihilating random walks.  That is, we show that 
\begin{equation}\label{ann}
\lim_{t \rightarrow \infty} \ P \,(\zeta_t(u) = 0) \ = \ 1 \quad \hbox{for all} \ u \in \D.
\end{equation}
We use different strategies to
 deal with active particles and frozen particles.
 We start by proving extinction of the active particles since it is one of the keys to showing extinction of the frozen particles,
 but the main ingredient to prove the latter is the result of Proposition \ref{prop:symmetry}.
 To see that (\ref{ann}) indeed implies Theorem \ref{thm}, we observe that for all $x, y \in \Z$ with $x < y$, we have
 $$ \begin{array}{rcl}
    \displaystyle \lim_{t \to \infty} \ P \,(\eta_t (x) \neq \eta_t (y)) & \leq &
    \displaystyle \lim_{t \to \infty} \ P \,(\eta_t (x + j - 1) \neq \eta_t (x + j) \ \hbox{for some} \ j = 1, 2, \ldots, y - x) \vspace{6pt} \\ \hspace{70pt} & \leq &
    \displaystyle \lim_{t \to \infty} \ P \,(\zeta_t (x + j - 1/2) \neq 0 \ \hbox{for some} \ j = 1, 2, \ldots, y - x) \\ \hspace{70pt} & \leq &
    \displaystyle \lim_{t \to \infty} \ \sum_{j = 1}^{y - x} \ P \,(\zeta_t (x + j - 1/2) \neq 0) \ = \ 0. \end{array} $$
 It remains to prove almost sure extinction of the systems of annihilating random walks, which is done in Lemma \ref{lem:active}
 for the active particles and in Lemma \ref{lem:frozen} for the frozen particles.
 Note that Lemmas \ref{lem:active} and \ref{lem:frozen} immediately imply (\ref{ann}).
 
\begin{lemma} --
\label{lem:limits}
 The limits
 $$ \lim_{t \to \infty} \ E \,(\zeta_t (u) \ \ind \{\zeta_t (u) \neq F \}) \quad \hbox{and} \quad
    \lim_{t \to \infty} \ E \,(\zeta_t (u) \ \ind \{\zeta_t (u) = F \}) $$
 exist and do not depend on the choice of $u \in \D$.
\end{lemma}
\begin{proof}
 First of all, since the initial configuration as well as the evolution rules of the process are translation invariant in space, the
 probability that $u$ has $i$ particles at time $t$ does not depend on the choice of a specific site $u \in \D$, i.e.
 $$ P \,(\zeta_t (u) = i) \ = \ P \,(\zeta_t (v) = i) \quad \hbox{for all} \ u, v \in \D. $$
 This implies that the limits superior and limits inferior of the expected values in the statements do not depend on $u$.
 In addition, since particles can only annihilate, the expected number of particles per site is a nonincreasing function of time so
 it has a limit as $t \to \infty$ and we write
\begin{equation}
\label{eq:limits-1}
 \limsup_{t \to \infty} \ E \,(\zeta_t (u)) \ = \ \liminf_{t \to \infty} \ E \,(\zeta_t (u)) \ = \ \lim_{t \to \infty} \ E \,(\zeta_t (u)) \ = \ L.
\end{equation}
 Now, seeking a contradiction, we assume that
\begin{equation}
\label{eq:limits-2}
 L_- \ = \ \liminf_{t \to \infty} \ E \,(\zeta_t (u) \ \ind \{\zeta_t (u) \neq F \})
     \ < \ \limsup_{t \to \infty} \ E \,(\zeta_t (u) \ \ind \{\zeta_t (u) \neq F \}) \ = \ L_+
\end{equation}
 and let $\ep > 0$ such that $5 \ep = L_+ - L_-$.
 In view of \eqref{eq:limits-1} there exists $t_0 < \infty$ such that
\begin{equation}
\label{eq:limits-3}
  L - \ep \ < \ E \,(\zeta_t (u)) \ < \ L + \ep \quad \hbox{for all} \ t > t_0.
\end{equation}
 Also, in view of \eqref{eq:limits-2} there exists an infinite sequence of times
 $$ t_0 \ < \ s_1 \ < \ t_1 \ < \ s_2 \ < \ t_2 \ < \ \cdots \ < \ s_j \ < \ t_j \ < \ \cdots \ < \ \infty $$
 such that for all integers $j > 0$, we have
\begin{equation}
\label{eq:limits-4}
   E \,(\zeta_{s_j} (u) \ \ind \{\zeta_{s_j} (u) \neq F \}) \ < \ L_- + \ep \quad \hbox{and} \quad
   E \,(\zeta_{t_j} (u) \ \ind \{\zeta_{t_j} (u) \neq F \}) \ > \ L_+ - \ep.
\end{equation}
 It directly follows from \eqref{eq:limits-3} and \eqref{eq:limits-4} that
 $$ \begin{array}{l}
     E \,(\zeta_{t_j} (u) \ \ind \{\zeta_{t_j} (u) = F \}) \ - \ E \,(\zeta_{s_j} (u) \ \ind \{\zeta_{s_j} (u) = F \}) \vspace{6pt} \\ \hspace{20pt} = \
     E \,(\zeta_{t_j} (u)) \ - \ E \,(\zeta_{s_j} (u)) \ + \
     E \,(\zeta_{s_j} (u) \ \ind \{\zeta_{t_j} (u) \neq F \}) \ - \ E \,(\zeta_{t_j} (u) \ \ind \{\zeta_{s_j} (u) \neq F \}) \vspace{6pt} \\ \hspace{20pt} < \
    (L + \ep) \ - \ (L - \ep) \ + \ (L_- + \ep) \ - \ (L_+ - \ep) \ = \ L_- \ - \ L_+ \ + \ 4 \ep \ = \ - \ep \end{array} $$
 for all $j > 0$ from which we deduce that
\begin{equation}
\label{eq:limits-5}
  \begin{array}{l}
    P \,(\zeta_{t_j} (u) = F) \ - \ P \,(\zeta_{s_j} (u) = F) \vspace{6pt} \\ \hspace{20pt} = \
    F^{-1} \ E \,(\zeta_{t_j} (u) \ \ind \{\zeta_{t_j} (u) = F \}) \ - \ F^{-1} \ E \,(\zeta_{s_j} (u) \ \ind \{\zeta_{s_j} (u) = F \}) \ < \ - \ep F^{-1}. \end{array}
\end{equation}
 Now, since each time an $F$-site becomes an $(F - 1)$-site, there are two particles that annihilate each other, the inequality
 in \eqref{eq:limits-5} also implies that
 $$ E \,(\zeta_{t_j} (u)) \ - \ E \,(\zeta_{s_j} (u)) \ \leq \ 2 \times P \,(\zeta_{t_j} (u) = F) \ - \ 2 \times P \,(\zeta_{s_j} (u) = F) \ < \ - 2 \ep F^{-1}. $$
 In particular, using once more that the expected number of particles per site is a nonincreasing function of time, and applying
 the previous inequality $F$ times, we obtain
 $$ \begin{array}{l}
    \displaystyle E \,(\zeta_{t_F} (u)) \ - \ E \,(\zeta_{s_1} (u)) \ \leq \
    \displaystyle \sum_{j = 1}^F \ \big( E \,(\zeta_{t_j} (u)) \ - \ E \,(\zeta_{s_j} (u)) \big) \ < \ - 2 \ep, \end{array} $$
 which contradicts \eqref{eq:limits-3} above.
 In particular, \eqref{eq:limits-2} does not hold so the expected number of active particles per site has a limit as $t \to \infty$.
 To prove that the second limit in the statement also exists, we simply observe that
 $$ E \,(\zeta_t (u) \ \ind \{\zeta_t (u) = F \}) \ = \ E \,(\zeta_t (u)) \ - \ E \,(\zeta_t (u) \ \ind \{0 < \zeta_t (u) < F \}) $$
 and invoke the fact that both expected values on the right-hand side have a limit as times goes to infinity.
 This completes the proof.
\end{proof}

\begin{lemma} --
\label{lem:active}
 There is extinction of the active particles, i.e.,
 $$ \lim_{t \to \infty} \,P \,(0 < \zeta_t (u) < F) \ = \ 0 \quad \hbox{for all} \ u \in \D. $$
\end{lemma}
\begin{proof}
 Note first that, by Lemma \ref{lem:limits}, the expected number of active particles per site does not depend on the choice of a
 specific site and has a limit as time goes to infinity.
 We assume by contradiction that this limit is positive, which implies that
\begin{equation}
\label{eq:active-1}
  E \,(\zeta_t (u) \ \ind \{0 < \zeta_t (u) < F \}) \ \geq \ M \ > \ 0 \quad \hbox{for all} \ t \geq 0.
\end{equation}
 We say that there is an event at site $u$ at time $t$ if one of the following occurs.
\begin{enumerate}
 \item Annihilating event: $\zeta_t (u) + \zeta_t (u + 1) = \zeta_{t-} (u) + \zeta_{t-} (u + 1) - 2$. \vspace{4pt}
 \item Freezing event: $\zeta_{t-} (u) = F - 1$ and $\zeta_t (u) = F$.
\end{enumerate}
 Denote by $\mathcal A_t (u)$ and $\mathcal F_t (u)$ the number of events at site $u$ by time $t$, and observe that the joint distribution
 of these two random variables does not depend on the specific choice of the site $u$.
 This follows by again invoking the translation invariance in space of the initial configuration and the evolution rules of the process.
 Since active particles evolve according to symmetric random walks run at a positive rate, they are doomed to either annihilation
 or freezing.  Hence, \eqref{eq:active-1} implies that
\begin{equation}
\label{eq:active-2}
  \lim_{t \to \infty} \ E \,(\mathcal A_t (u)) \ + \ \lim_{t \to \infty} \ E \,(\mathcal F_t (u)) \ = \ \infty.
\end{equation}
 Now, we observe that, since the initial number of particles per site is bounded by $F$ and each annihilating event removes two particles
 from the system,
\begin{equation}
\label{eq:active-3}
  E \,(\mathcal A_t (u)) \ \leq \ F / 2 \ < \ \infty \quad \hbox{for all} \ t \geq 0.
\end{equation}
 In addition, two consecutive freezing events at site $u$ must be separated by an annihilating event at either site $u$ or site
 $u - 1$ which gives the upper bound
\begin{equation}
\label{eq:active-4}
  E \,(\mathcal F_t (u)) \ \leq \ 1 + E \,(\mathcal A_t (u)) \ \leq \ 1 + F / 2 \ < \ \infty \quad \hbox{for all} \ t \geq 0.
\end{equation}
 Since the combination of \eqref{eq:active-3} and \eqref{eq:active-4} contradicts \eqref{eq:active-2}, the lemma follows.
\end{proof}

\begin{lemma} --
\label{lem:frozen}
 There is extinction of the frozen particles, i.e.,
 $$ \lim_{t \to \infty} \,P \,(\zeta_t (u) = F) \ = \ 0 \quad \hbox{for all} \ u \in \D. $$
\end{lemma}
\begin{proof}
 Note as previously that, by Lemma \ref{lem:limits}, the limit to be estimated exists and does not depend on the specific choice
 of site $u \in \D$.
 Let $\ep > 0$.
 By Lemma \ref{lem:active}, there is $t_0 < \infty$ large such that
\begin{equation}
\label{eq:frozen-1}
  E \,(\zeta_t (u) \ \ind \{0 < \zeta_t (u) < F \}) \ < \ \frac{\ep}{2} \quad \hbox{for all} \ t \geq t_0.
\end{equation}
 To prove extinction of the frozen particles, we apply successively Proposition \ref{prop:symmetry} to obtain the existence of an
 increasing sequence of times $\{t_j : j \geq 0 \}$ tending to infinity such that
 $$ P \,(\zeta_s (u) = F \ \hbox{for all} \ s \in (t_j, t_{j + 1}) \ | \ \zeta_{t_j} (u) = F) \ < \ \frac{1}{2} \quad \hbox{for all} \ j \geq 0. $$
 In particular, since each time an $F$-site becomes an $(F - 1)$-site, there are two particles that annihilate each other,
 the previous inequality implies that
\begin{equation}
\label{eq:frozen-2}
  E \,(\zeta_{t_j} (u)) \ - \ E \,(\zeta_{t_{j + 1}} (u)) \ > \
  2 \times \bigg(1 - \frac{1}{2} \bigg) \ P \,(\zeta_{t_j} (u) = F) \ = \ P \,(\zeta_{t_j} (u) = F)
\end{equation}
 for all integers $j \geq 0$.
 To see this, note that if a site is frozen at time $t_j$, then with probability at least $1/2$ this site will be visited by an active particle before time $t_{j+1}$, resulting in the annihilation of two particles.  Therefore, the expected number of
 particles killed per site between times $t_j$ and $t_{j+1}$ is larger than the probability of a site being an $F$-site at time $t_j$, which gives \eqref{eq:frozen-2}.
 Now, using \eqref{eq:frozen-2} and then \eqref{eq:frozen-1}, we deduce that
 $$ \begin{array}{l}
     F \times P \,(\zeta_{t_{j + 1}} (u) = F) \ = \
     E \,(\zeta_{t_{j + 1}} (u) \ \ind \{\zeta_{t_{j + 1}} (u) = F \}) \ \leq \
     E \,(\zeta_{t_{j + 1}} (u)) \vspace{6pt} \\ \hspace{15pt} \leq \
     E \,(\zeta_{t_j} (u)) \ - \ P \,(\zeta_{t_j} (u) = F) \vspace{6pt} \\ \hspace{15pt} = \
     E \,(\zeta_{t_j} (u) \ \ind \{\zeta_{t_j} (u) = F \}) \ + \
     E \,(\zeta_{t_j} (u) \ \ind \{\zeta_{t_j} (u) \neq F \}) \ - \ P \,(\zeta_{t_j} (u) = F) \vspace{6pt} \\ \hspace{15pt} \leq \
     F \ P \,(\zeta_{t_j} (u) = F) \ + \ \ep / 2 \ - \ P \,(\zeta_{t_j} (u) = F) \ = \
    (F - 1) \ P \,(\zeta_{t_j} (u) = F) \ + \ \ep / 2. \end{array} $$
 In particular, a simple induction gives
 $$ \begin{array}{rcl}
    \displaystyle P \,(\zeta_{t_n} (u) = F) & \leq &
    \displaystyle \bigg(1 - \frac{1}{F} \bigg) \ P \,(\zeta_{t_{n - 1}} (u) = F) \ + \ \frac{\ep}{2F} \vspace{4pt} \\ & \leq &
    \displaystyle \bigg(1 - \frac{1}{F} \bigg)^n \ P \,(\zeta_{t_0} (u) = F) \ + \ \frac{\ep}{2F} \ \sum_{j = 0}^{n - 1} \ \bigg(1 - \frac{1}{F} \bigg)^j \vspace{4pt} \\ & \leq &
    \displaystyle \bigg(1 - \frac{1}{F} \bigg)^n + \ \frac{\ep}{2F} \ \frac{1}{1 - (1 - 1/F)} \ = \
    \displaystyle \bigg(1 - \frac{1}{F} \bigg)^n + \ \frac{\ep}{2} \ \leq \ \ep \end{array} $$
 for sufficiently large $n$.  In view of Lemma \ref{lem:limits}, the previous inequality implies that
 $$ \lim_{t \to \infty} \ P \,(\zeta_t (u) = F) \ = \ \lim_{n \to \infty} \ P \,(\zeta_{t_n} (u) = F) \ \leq \ \ep \quad \hbox{for all} \ u \in \D. $$
 Since this holds for all $\ep > 0$ arbitrarily small, the lemma follows.
\end{proof}

\end{document}